\documentclass[letter,11pt]{article}
\usepackage[pagewise]{lineno}

\usepackage[margin=1in]{geometry}
\usepackage[useregional]{datetime2}
\usepackage[english]{babel}
\usepackage[utf8]{inputenc}
\usepackage[T1]{fontenc}
\usepackage{xcolor, graphicx,changepage,xcolor, float, array, amssymb, amsmath, mathrsfs, amsfonts, hyperref, url, amsthm, tikz-cd, fancyhdr, enumerate, bbm}
\makeatletter
\newcommand*{\rom}[1]{\expandafter\@slowromancap\romannumeral #1@}
\makeatother

\makeatletter
\renewcommand*\env@matrix[1][*\c@MaxMatrixCols c]{%
	\hskip -\arraycolsep
	\let\@ifnextchar\new@ifnextchar
	\array{#1}}
\makeatother

\theoremstyle{definition}
\newtheorem{theorem}{Theorem}[section]
\newtheorem{lemma}[theorem]{Lemma}
\newtheorem{definition}[theorem]{Definition}

\newtheorem{corollary}[theorem]{Corollary}

\theoremstyle{remark}
\newtheorem{remark}[theorem]{Remark}

\usepackage{setspace}
\title{Van der Corput's difference theorem for amenable groups and the left regular representation}
\author{Sohail Farhangi \\ Department of mathematics and computer science, University of Adam Mickiewicz,\\
Pozna\'n, 61-712, Greater Poland Voivodeship, Poland\\ email address: sohail.farhangi@gmail.com}

\begin{document}

\maketitle

\begin{abstract}
    We establish a connection between two variants of van der Corput's Difference Theorem (vdCDT) for countably infinite amenable groups $G$ and the ergodic hierarchy of mixing properties of a unitary representation $U$ of $G$. In particular, we show that one variant of vdCDT corresponds to subrepresentations of the left regular representation, and another variant of vdCDT corresponds to the absence of finite dimensional subrepresentations. We then obtain applications for measure preserving actions of countably infinite abelian groups. 
\end{abstract}
Math subject classifications: 22D40, 37A30, 43A35\\
Keywords: F\o lner sequence, disjointness, ultraproduct, ergodic theorem, Lebesgue spectrum, weak mixing
\section{Introduction and statements of results}

\subsection{Introduction}
Bergelson in \cite{WMPET} introduced the following forms of van der Corput's Difference Theorem.

\begin{theorem}
\label{ClassicalvanderCorput'sDifferenceTheorems}
Let $\mathcal{H}$ be a Hilbert space and $(x_n)_{n = 1}^{\infty} \subseteq \mathcal{H}$ a bounded sequence of vectors. 

\begin{enumerate}
    \item[(i)] If for every $h \in \mathbb{N}$ we have

    \begin{equation}
        \lim_{N\rightarrow\infty}\frac{1}{N}\sum_{n = 1}^N\langle x_{n+h}, x_n\rangle = 0\text{, then }\lim_{N\rightarrow\infty}\left|\left|\frac{1}{N}\sum_{n = 1}^Nx_n\right|\right| = 0.
    \end{equation}
    
    \item[(ii)] If

    \begin{equation}
        \lim_{h\rightarrow\infty}\limsup_{N\rightarrow\infty}\left|\frac{1}{N}\sum_{n = 1}^N\langle x_{n+h}, x_n\rangle\right| = 0\text{, then }\lim_{N\rightarrow\infty}\left|\left|\frac{1}{N}\sum_{n = 1}^Nx_n\right|\right| = 0.
    \end{equation}
    
    \item[(iii)] If
    
    \begin{equation}
        \lim_{H\rightarrow\infty}\frac{1}{H}\sum_{h = 1}^H\limsup_{N\rightarrow\infty}\left|\frac{1}{N}\sum_{n = 1}^N\langle x_{n+h}, x_n\rangle\right| = 0\text{, then }\lim_{N\rightarrow\infty}\left|\left|\frac{1}{N}\sum_{n = 1}^Nx_n\right|\right| = 0.
    \end{equation} 
\end{enumerate}
\end{theorem}

In \cite{LebesgueSpectrumvanderCorput} the following generalization of Theorem \ref{ClassicalvanderCorput'sDifferenceTheorems}(i) was proven by the author. For the sake of presentation, we defer the precise definitions of spectrally Lebesgue and spectrally singular sequences to Section 3. For now, it suffices to know that constant sequences and almost periodic sequences are spectrally singular.

\begin{theorem}\label{LebesgueSpectrumvdCForZ}
Let $\mathcal{H}$ be a Hilbert space and $(x_n)_{n = 1}^{\infty} \subseteq \mathcal{H}$ a bounded sequence of vectors. If

    \begin{equation}\label{AssumptionEquationForZvdC}
        \sum_{h = 1}^\infty\limsup_{N\rightarrow\infty}\left|\frac{1}{N}\sum_{n = 1}^N\langle x_{n+h}, x_n\rangle\right|^2 < \infty,
    \end{equation}
    then $(x_n)_{n = 1}^\infty$ is a spectrally Lebesgue sequence. In particular, if $(c_n)_{n = 1}^\infty \subseteq \mathbb{C}$ is bounded and spectrally singular, then

    \begin{equation}
        \lim_{N\rightarrow\infty}\left|\left|\frac{1}{N}\sum_{n = 1}^Nc_nx_n\right|\right| = 0.
    \end{equation}
    Furthermore, if $\mathcal{H} = L^2(X,\mu)$, and $(y_n)_{n = 1}^\infty \subseteq L^\infty(X,\mu)$ is bounded and spectrally singular, then

    \begin{equation}
        \lim_{N\rightarrow\infty}\left|\left|\frac{1}{N}\sum_{n = 1}^Ny_nx_n\right|\right| = 0.
    \end{equation}
\end{theorem}
In \cite{LebesgueSpectrumvanderCorput}, Theorem \ref{LebesgueSpectrumvdCForZ} was used to prove various ergodic theorems for noncommuting $\mathbb{Z}$-actions. One of the main goals of this article is to generalize Theorem \ref{LebesgueSpectrumvdCForZ} to the setting of countably infinite amenable groups. We will then obtain applications to measure preserving actions of countably infinite abelian groups, so we only state here the following result that is a special case of what follows from Theorems \ref{LebesgueSpectrumAmenablevdC} and \ref{StrongDisjointnessOfSequences} and Corollary \ref{StrongDisjointnessOfSequencesCorollary}.

\begin{theorem}\label{StatementOfGeneralvdCInIntro}
    Let $G$ be a countably infinite abelian group, let $(F_n)_{n = 1}^\infty$ be a F\o lner sequence in $G$, let $\mathcal{H}$ be a Hilbert space, and let $u:G\rightarrow\mathcal{H}$ be a bounded map. If

    \begin{equation}\label{AssumptionEquationForAbelianvdC}
        \sum_{h \in G}\limsup_{N\rightarrow\infty}\left|\frac{1}{|F_N|}\sum_{g \in F_N}\langle u(g+h),u(g)\rangle\right|^2 < \infty,
    \end{equation}
    then $u$ is spectrally Lebesgue. In particular, if $c:G\rightarrow\mathbb{C}$ is a bounded spectrally singular map, then

    \begin{equation}
        \lim_{N\rightarrow\infty}\left|\left|\frac{1}{|F_N|}\sum_{F_N}c(g)u(g)\right|\right| = 0.
    \end{equation}
    Furthermore, if $\mathcal{H} = L^2(X,\mu)$, and $y:G\rightarrow L^\infty(X,\mu)$ is bounded and spectrally singular, then
    
    \begin{equation}
        \lim_{N\rightarrow\infty}\left|\left|\frac{1}{|F_N|}\sum_{F_N}y(g)u(g)\right|\right| = 0.
    \end{equation}
\end{theorem}
In \cite[Chapter 2]{SohailsPhDThesis} results were proven that generalize Theorems \ref{ClassicalvanderCorput'sDifferenceTheorems}(ii)-(iii) in the same way that Theorem \ref{LebesgueSpectrumvdCForZ} generalizes Theorem \ref{ClassicalvanderCorput'sDifferenceTheorems}(i). In particular, it was shown that Theorem \ref{ClassicalvanderCorput'sDifferenceTheorems}(ii) is associated with the notion of strong mixing from the ergodic hierarchy of mixing properties, and Theorem \ref{ClassicalvanderCorput'sDifferenceTheorems}(iii) is associated with the notion of weak mixing. Another goal of this article is to establish a connection between a version of Theorem \ref{ClassicalvanderCorput'sDifferenceTheorems}(iii) for amenable groups and the notion of weak mixing. While we could proceed as in \cite[Chapter 2]{SohailsPhDThesis} and establish similar connections between other variants of van der Corput's Difference Theorem in amenable groups and different levels of the ergodic hierarchy of mixing, such as ergodicity, mild mixing, and strong mixing, we do not do this for the sake of brevity. Connections between variations of van der Corput's Difference Theorem for semigroups and mixing properties had previously been investigated in \cite{UnifiedVDC}.

\subsection{Notation}
We will always use $G$ to denote a group $(G,\cdot)$ whose identity element is given by $e_G$. If $G$ is abelian, we will emphasize this by denoting the group operation by $+$, otherwise we write $g_1g_2 := g_1\cdot g_2$. All groups considered in this article will be countably infinite discrete groups. When $G$ is also an abelian group, it will have a Pontryagin dual that we denote by $\widehat{G}$. We will freely use facts about Pontryagin duality throughout this article and refer the reader to \cite{ACourseInAbstractHarmonicAnalysisByFolland} for further details. We will also be dealing with \textit{amenable groups}, which are those groups $G$ possesing a \textit{right-F\o lner sequence}. A right-F\o lner sequence $(F_n)_{n = 1}^\infty$ consists of finite subsets of $G$ for which

\begin{equation}
    \lim_{n\rightarrow\infty}\frac{|F_n\triangle F_ng|)}{|F_n|} = 0,
\end{equation}
for all $g \in G$. It is well known that all solvable (and hence abelian) groups are amenable.

If $(X,\mathscr{B},\mu)$ is a standard probability space and $G$ is a group, then a measure preserving action of $G$ is a collection of invertible bi-measurable maps $T_g:X\rightarrow X$ for $g \in G$, satisfying the following conditions:

\begin{enumerate}[(i)]
    \item $T_{e_G} = \text{Id}$.
    \item $T_gT_h = T_{gh}$ for $g,h \in G$,
    \item $\mu(A) = \mu(T_gA)$ for all $A \in \mathscr{B}$.
\end{enumerate} The tuple $(X,\mathscr{B},\mu,(T_g)_{g \in G})$ is a \textit{(measure preserving) $G$-system}. If for any $A \in \mathscr{B}$ satisfying $\mu(A\triangle T_gA) = 0$ for all $g \in G$ we have $\mu(A) \in \{0,1\}$, then the action of $G$ is \textit{ergodic}. The action of $G$ is \textit{totally ergodic}\footnote{As we will see later on, this is not a standard definition. Nonetheless, it coincides with the standard definition of total ergodicity when $G = \mathbb{Z}$.} if for any finite index subgroup $H \le G$, the action $(T_g)_{g \in H}$ is ergodic. The action of $G$ is \textit{weakly mixing} if the action of $(T_g\times T_g)_{g \in G}$ on $(X\times X,\mathscr{B}\otimes\mathscr{B},\mu\times\mu)$ is ergodic. We let $\mathcal{I}_T \subseteq \mathscr{B}$ denote the $\sigma$-algebra of sets that are invariant under the action $(T_g)_{g \in G}$. We let $\mathbb{E}[\cdot|\mathcal{I}_T]:L^2(X,\mathscr{B},\mu)\rightarrow L^2(X,\mathcal{I}_T,\mu)$ denote the orthogonal projection. Now let us assume that $G$ is abelian. For any $f \in L^2(X,\mu)$, there exists a finite positive measure $\mu_{T,f}$ on $\widehat{G}$ called the spectral measure of $f$ (with respect to $T$) for which $\langle T_gf,f\rangle = \hat{\mu}_{T,f}(g)$. There exists a finite positive measure $\mu_T$ on $\widehat{G}$ for which $\mu_{T,f} \ll \mu_T$ for all $f \in L^2_0(X,\mu)$\footnote{We restrict ourselves to those $f \in L^2(X,\mu)$ for which $\int_Xfd\mu = 0$ so that the system $(X,\mathscr{B},\mu,(T_g)_{g \in G})$ is ergodic if and only if $\mu_T(\{0\}) = 0$.}. The equivalence class $[\mu_T]$ of measures that are mutually absolutely continuous with respect to $\mu_T$ is the maximal spectral type of $T$. If $\lambda'$ denotes the Haar measure of $\widehat{G}$, then we say that $T$ has \textit{singular spectrum} if $[\mu_T]\perp\lambda'$, and $T$ has \textit{Lebesgue spectrum} if $\lambda' \in [\mu_T]$. Examples of $\mathbb{Z}$-systems with singular spectrum as well as a brief discussion about how generic such systems are is given in \cite[Remark 1.16]{LebesgueSpectrumvanderCorput}.

When the group $G$ is a general (not necessarily abelian) amenable group, the notions of Lebesgue spectrum and singular spectrum can be defined through the left regular representation of $G$. In Section 2 we will present the prerequisite notation and theory needed to discuss this in further detail.

\begin{remark}\label{RemarkOnNonCompactness}
    We work with non-compact groups because those are the groups that require averaging along F\o lner sequences. In particular, if $G$ was a compact group, we would have $\lambda(G) = 1$ where $\lambda$ is the Haar measure, so we could just average over the entire group rather than a F\o lner sequence. Since our $G$ will be non-compact, we will always have that $\lim_{n\rightarrow\infty}|F_n| = \infty$ for any F\o lner sequence $(F_n)_{n = 1}^\infty$.
    
    Furthermore, the Peter-Weyl Theorem tells us that every strongly continuous irreducible representation of a compact group is a finite dimensional subrepresentation of the left regular representation. The applications of our generalizations of van der Corput's Difference Theorem are based on disjointness of representations with all finite dimensional representations of $G$, or disjointness of representations from the left regular representation of $G$. If $G$ is compact, then there are no representations disjoint from the finite dimensional representations, and there are no representations disjoint from the left regular representation, so our applications would be vacuous in this setting.

    We restrict our attention to countable discrete groups to avoid difficulties that arise when checking that certain functions on the group are continuous in Section 3. 
\end{remark}

\subsection{Background}
In \cite{FurstenbergMultipleRecurrence} Furstenberg gave a new ergodic theoretic proof of the celebrated Szemer\'edi Theorem by showing that it is a corollary of the following result, which is now known as the Furstenberg multiple recurrence theorem.

\begin{theorem}\label{MultipleRecurrenceTheorem}
Let $(X,\mathscr{B},\mu,(T^n)_{n \in \mathbb{Z}})$ be a $\mathbb{Z}$-system and $B \in \mathscr{B}$ with $\mu(B) > 0$. For any integer $k > 1$ we have

\begin{equation}\label{FurstenbergMultipleRecurrenceEquation}
    \liminf_{N\rightarrow\infty}\frac{1}{N}\sum_{n = 1}^N\mu\left(B\cap T^{-n}B\cap T^{-2n}B\cap\cdots\cap T^{-(k-1)n}B\right) > 0.
\end{equation}
\end{theorem}
The multiple recurrence theorem has produced a lot of interest in the study of multiple ergodic averages, which are expressions of the form

\begin{equation}\label{MultipleErgodicAverageEquation}
    \lim_{N\rightarrow\infty}\frac{1}{N}\sum_{n = 1}^NT_{a_1(n)}f_1T_{a_2(n)}f_2\cdots T_{a_\ell(n)}f_\ell,
\end{equation}
where $(X,\mathscr{B},\mu,(T_g)_{g \in G})$ is a $G$-system, $f_1,\cdots,f_\ell \in L^\infty(X,\mu)$, and  $a_1,\cdots,a_\ell:\mathbb{N}\rightarrow G$. We note that the existence and the value of the limit of multiple ergodic averages is not strictly needed to deduce combinatorial corollaries, as Furstenberg only needed a single $n$ for which the summand in Equation \eqref{FurstenbergMultipleRecurrenceEquation} was positive in order to deduce Szemere\'edi's Theorem. Nonetheless, it has also been a topic of interest that has been investigated in papers such as \cite{StartOfConvergenceOfMultipleErgodicAverages,ConvergenceOfMultipleErgodicAveragesForCommutingTransformations,OriginalGeneralConvergenceOfMultipleErgodicAverages,ZieglerCharacteristicFactor,WalshErgodicTheorem}. Furstenberg's multiple recurrence theorem was extended from $\mathbb{Z}$-systems to $\mathbb{Z}^d$-systems by Furstenberg and Katznelson \cite{CommutingSzemeredi}, and a polynomial version of this latter result, now known as the Polynomial Szemer\'edi Theorem, was proven by Bergelson and Leibman \cite{PolynomialSzemeredi}.  Leibman \cite{LeibmanNilpotentPolynomialSzemeredi} extended the Polynomial Szemer\'edi Theorem to the setting of nilpotent groups. On the other hand, Bergelson and Leibman \cite{NoSolvableRoth} showed that for any solvable group $G$, there exists a $G$-system for which the multiple ergodic averages do not behave well. The fact that some commutativity assumptions on the group $G$ are needed in order to fruitfully study multiple ergodic averages was already known to Furstenberg \cite[Page 40]{FurstenbergBook} and Berend \cite[Example 7.1]{Berend'sExample}. Building upon these ideas, we have the following result of Frantzikinakis, Lesigne, and Wierdl.

\begin{theorem}[{\cite[Lemma 4.1]{RandomSequencesAndPtwiseConv}}]\label{DoubleBernoulliIsNutz}
    Let $a,b:\mathbb{N}\rightarrow\mathbb{Z}\setminus\{0\}$ be injective sequences and $F$ be any subset of $\mathbb{N}$. Then there exist a probability space $(X,\mathscr{B},\mu)$, measure preserving automorphisms $T,S:X\rightarrow X$, both of them Bernoulli, and $A \in \mathscr{B}$, such that

    \begin{equation}
        \mu\left(T^{-a(n)}A\cap S^{-b(n)}A\right) = \begin{cases}
                   0&\text{if }n \in F,\\
                   \frac{1}{4}&\text{if }n \notin F.
                \end{cases}
    \end{equation}
\end{theorem}

While the previous result shows that ergodic averages can behave wildly when there are no commutativity assumptions, the next result of Frantzikinakis shows that the commutativity assumptions can be replaced by structural assumptions.

\begin{theorem}[{\cite[Corollary 1.7]{FSystemsOfHardySequences}}]
\label{FrantzikinakisHardyFieldResult}
Let $a:\mathbb{R}_+\rightarrow\mathbb{R}$ be a Hardy field function\footnote{The notion of Hardy field function will not appear again in this article, so we do not define it. It suffices to note that Theorem \ref{FrantzikinakisHardyFieldResult} applies to sequences such as $a(t) = t^{\alpha}\ln^\beta(t)$ for $1 < \alpha \in \mathbb{R}\setminus\mathbb{N}$ and $\beta \in \mathbb{R}$.} for which there exist some $\epsilon > 0$ and $d \in \mathbb{Z}_+$ satisfying

\begin{equation}
    \lim_{t\rightarrow\infty}\frac{a(t)}{t^{d+\epsilon}} = \lim_{t\rightarrow\infty}\frac{t^{d+1}}{a(t)} = \infty.
\end{equation}
Furthermore, let $(X,\mathscr{B},\mu)$ be a probability space and $T,S:X\rightarrow X$ be measure preserving automorphisms, and suppose that $T$ has zero entropy. Then
\begin{enumerate}[(i)]
    \item For every $f,g \in L^\infty(X,\mu)$ we have
    
    \begin{equation}
        \lim_{N\rightarrow\infty}\frac{1}{N}\sum_{n = 1}^NT^nf\cdot S^{\lfloor a(n)\rfloor}g = \mathbb{E}[f|\mathcal{I}_T]\mathbb{E}[g|\mathcal{I}_S],
    \end{equation}
    where the limit is taken in $L^2(X,\mu)$. 
    
    \item For every $A \in \mathscr{B}$ we have
    
    \begin{equation}
        \lim_{N\rightarrow\infty}\frac{1}{N}\sum_{n = 1}^N\mu\left(A\cap T^{-n}A\cap S^{-\lfloor a(n)\rfloor}A\right) \ge \mu(A)^3.
    \end{equation}
\end{enumerate}
\end{theorem}
See \cite{AZeroEntropyCounterExampleWithoutCommutativity,MultipleRecurrenceForNoncommutingTransformationsAlongRationallyIndependentPolynomials,LebesgueSpectrumvanderCorput,PolynomialFurstenbergJoinings} for more results about multiple ergodic averages involving non-commuting transformations. In this article we will also prove results about ergodic averages of noncommuting actions $T$ and $S$ of an amenable group $G$. We will be assuming that the action $T$ has singular spectrum instead of zero entropy. Rohlin showed in \cite{SingularSystemsHaveZeroEntropy} that when $G = \mathbb{Z}$, every system with singular spectrum also has zero entropy, and it is known that the horocycle flow is an example of a zero entropy $\mathbb{Z}$-system \cite{HorocycleFlowHasZeroEntropy} that has Lebesgue spectrum \cite{HorocycleFlowHasLebesgueSpectrum}. When $G$ is a countable amenable group, the fact that systems with singular spectrum also have zero entropy can be deduced as a corollary of the work of Dooley and Golodets \cite{LebesgueSpectrumForKMixingAmenableGroupActions} as well as Danilenko and Park \cite{SinaiFactorTheoremForCountableAmenableGroups}.

Unlike the previously mentioned results, our ergodic averages will be taken over F\o lner sequences $(F_n)_{n = 1}^\infty$ of $G$ instead of $[1,N]$ because we will be considering maps $a_1,\cdots,a_\ell:G\rightarrow G$ instead of $a_1,\cdots,a_\ell:\mathbb{N}\rightarrow G$.

\subsection{Applications to measure preserving systems}
In this subsection we state our main results for measure preserving actions of countably infinite abelian groups. While the results that we prove in Sections 3 and 4 apply more generally to amenable groups, we do not state them here, for the sake of simplicity. We begin by recalling the notion of joint ergodicity, which was first investigated in \cite{JointlyErgodicMeasurePreservingTransformations}, but our definitions will be motivated by those appearing in \cite{JointErgodicityForLCAGroups} and \cite{JointErgodicityOfSequences}.
\begin{definition}
    Let $G$ be an amenable group and $(F_n)_{n = 1}^\infty$ a right-F\o lner sequence in $G$. We say that a collection of mappings $a_1,\cdots,a_\ell:G\rightarrow G$ is 
    \begin{enumerate}[(i)]
        \item \textit{jointly ergodic (along $(F_n)_{n = 1}^\infty)$ for the $G$-system $(X,\mathscr{B},\nu,(T_g)_{g \in G})$} if, for all functions $f_1,\cdots,f_\ell \in L^\infty(X,\mu)$, we have 

        \begin{equation}
            \lim_{N\rightarrow\infty}\frac{1}{|F_N|}\sum_{g \in F_N}T_{a_1(g)}f_1T_{a_2(g)}f_2\cdots T_{a_\ell(g)}f_\ell = \prod_{j = 1}^\ell\int_Xf_jd\nu,
        \end{equation}
        with convergence taking place in $L^2(X,\nu)$.
    \item \textit{jointly ergodic (along $(F_n)_{n = 1}^\infty)$} if it is jointly ergodic for every ergodic $G$-system.

    \item \textit{jointly ergodic (along $(F_n)_{n = 1}^\infty)$ for totally ergodic systems} if it is jointly ergodic for every totally ergodic $G$-system.
    \end{enumerate}
\end{definition}
The next result follows from Theorem \ref{StatementOfGeneralvdCInIntro}, and is the basis for all of our subsequent applications. 
\begin{theorem}[cf. Theorem \ref{GeneralPrincipleForAmenableGroups}]\label{GeneralPrincipleForLCAGroups}
    Let $G$ be a countably infinite abelian group and $(F_n)_{n = 1}^\infty$ a F\o lner sequence in $G$. Let $a_1,\cdots,a_\ell:G\rightarrow G$ be a collection of mappings for which $A(h) := \{a_i(\cdot)\}_{i = 1}^\ell\cup\{a_i(\cdot+h)\}_{i = 1}^\ell$ is jointly ergodic for every $h \in G\setminus\{e_G\}$. Let $(X,\mathscr{B},\nu)$ be a probability space and $T_g,S_g:X\rightarrow X$ be (not necessarily commuting) measure preserving actions of $G$ for which $(T_g)_{g \in G}$ has singular spectrum and $(S_g)_{g \in G}$ is ergodic. For any $f_0,f_1,\cdots,f_\ell \in L^\infty(X,\nu)$ we have

    \begin{equation}\label{GeneralConsequenceOfLebesgueSpectrumvdC}
        \lim_{N\rightarrow\infty}\frac{1}{|F_N|}\sum_{g \in F_N}T_gf_0S_{a_1(g)}f_1S_{a_2(g)}f_2\cdots S_{a_\ell(g)}f_\ell = \mathbb{E}[f_0|\mathcal{I}_T]\prod_{j = 1}^\ell\int_Xf_jd\nu,
    \end{equation}
    with convergence taking place in $L^2(X,\nu)$. Furthermore, if $A(h)$ is only assumed to be jointly ergodic for totally ergodic systems, then Equation \eqref{GeneralConsequenceOfLebesgueSpectrumvdC} still holds when $(S_g)_{g \in G}$ is totally ergodic.
\end{theorem}

Our first two applications of Theorem \ref{GeneralPrincipleForLCAGroups} use results about joint ergodicity of polynomial actions from \cite{JointErgodicityForLCAGroups}, which we have stated later on as Theorems \ref{BestMoragues1} and \ref{BestMoragues2}.

\begin{theorem}\label{Application1Intro}
    Let $K$ be a countably infinite field with characteristic $0$. Let $(X,\mathscr{B},\nu)$ be a probability space and $T_g,S_g:X\rightarrow X$ measure preserving actions of $(K,+)$ for which the action $(T_g)_{g \in K}$ has singular spectrum and the action $(S_g)_{g \in K}$ is ergodic. Let $(F_n)_{n = 1}^\infty$ be a F\o lner sequence in $(K,+)$ and $\ell \in \mathbb{N}$. Let $p_1,\cdots,p_\ell \in K[x]$ be polynomials for which $\text{deg}(p_1) \ge 2$ and $\text{deg}(p_{i+1}) \ge 2+\text{deg}(p_i)$ for $1 \le i < \ell$. Then for any $f_0,f_1,\cdots,f_\ell \in L^\infty(X,\mu)$ we have

    \begin{equation}
        \lim_{N\rightarrow\infty}\frac{1}{\left|F_N\right|}\sum_{n \in F_N}T_nf_0\prod_{j = 1}^\ell S_{p_j(n)}f_j = \mathbb{E}[f_0|\mathcal{I}_T]\prod_{j = 1}^\ell\int_Xf_jd\nu
    \end{equation}
    with convergence taking place in $L^2(X,\nu)$.
\end{theorem}

\begin{definition}
    Let $R$ be a commutative ring. We say that $R$ is \textit{good} if it is a countably infinite integral domain with characteristic $0$, and every non-zero ideal has finite index in $R$.
\end{definition}

\begin{theorem}\label{Application2Intro}
    Let $R$ be a good ring. Let $(X,\mathscr{B},\nu)$ be a probability space and $T_g,S_g:X\rightarrow X$ measure preserving actions of $(R,+)$ for which $(T_g)_{g \in R}$ has singular spectrum and $(S_g)_{g \in R}$ is totally ergodic. Let $(F_n)_{n = 1}^\infty$ be a F\o lner sequence in $(R,+)$ and $\ell \in \mathbb{N}$. Let $p_1,\cdots,p_\ell \in R[x]$ be polynomials for which $\text{deg}(p_1) \ge 2$ and $\text{deg}(p_{i+1}) \ge 2+\text{deg}(p_i)$ for $1 \le i < \ell$. Then for any $f_0,f_1,\cdots,f_k \in L^\infty(X,\mu)$ we have

    \begin{equation}
        \lim_{N\rightarrow\infty}\frac{1}{\left|F_N\right|}\sum_{n \in F_N}T_nf_0\prod_{j = 1}^\ell S_{p_j(n)}f_j = \mathbb{E}[f_0|\mathcal{I}_T]\prod_{j = 1}^\ell\int_Xf_jd\nu
    \end{equation}
    with convergence taking place in $L^2(X,\mu)$.
\end{theorem}

Our next application is a generalization of \cite[Theorem 1.8]{LebesgueSpectrumvanderCorput} from $\mathbb{Z}$-actions to actions of a countable abelian group.

\begin{theorem}[cf. Theorem \ref{Application3Main}]\label{Application3Intro}
    Let $G$ be a countably infinite abelian group and $(F_n)_{n = 1}^\infty$ a F\o lner sequence in $G$. Let $(X,\mathscr{B},\nu)$ be a probability space and $T_g,S_g:X\rightarrow X$ measure preserving actions of $G$ for which the action $(T_g)_{g \in G}$ has singular spectrum. Let $a:G\rightarrow G$ be a sequence, such that for every $\chi \in \widehat{G}$ and $h \in G\setminus\{e_G\}$ we have
    \begin{equation}\label{UniformDistributionOfDifferencesAssumptionEquation}
        \lim_{N\rightarrow\infty}\frac{1}{|F_N|}\sum_{g \in F_N}\chi(a(g+h)-a(g)) = 0,
    \end{equation}
    which is equivalent to requiring that $(\chi(a(g+h)-a(g)))_{g \in G}$ be uniformly distributed (with respect to $(F_n)_{n = 1}^\infty$) in $\overline{\{\chi(g)\ |\ g \in G\}}$, for all $\chi \in \widehat{G}$ and $h \in G\setminus\{e_G\}$.
    \begin{enumerate}[(i)]
        \item For any $f_1,f_2 \in L^\infty(X,\nu)$ we have

        \begin{equation}
            \lim_{N\rightarrow\infty}\frac{1}{|F_N|}\sum_{g \in F_N}T_gf_1S_{a(g)}f_2 = \mathbb{E}[f_1|\mathcal{I}_T]\mathbb{E}[f_2|\mathcal{I}_S]
        \end{equation}
        with convergence taking place in $L^2(X,\nu)$.

        \item If $A \in \mathscr{B}$, then

        \begin{equation}
            \lim_{N\rightarrow\infty}\frac{1}{|F_N|}\sum_{F_N}\mu\left(A\cap T_gA\cap S_{a(g)}A\right) \ge \mu(A)^3.
        \end{equation}

        \item Now assume that $G$ is finitely generated and that Equation \eqref{UniformDistributionOfDifferencesAssumptionEquation} only holds for $\chi \in \widehat{G}$ that do not vanish on any subgroup $H \le G$ of finite index. Then (i) and (ii) still hold when the action $(S_g)_{g \in G}$ is totally ergodic.
    \end{enumerate}
\end{theorem}

Let us now consider some examples of results that follow from Theorem \ref{Application3Intro}. Since the case of $G = \mathbb{Z}$ was addressed in \cite[Page 6]{LebesgueSpectrumvanderCorput}, we skip this case here. Next, we consider $G = \mathbb{Z}^d$. When we identify $\mathbb{Z}^d$ as the additive group of the ring of integers $R$ of an algebraic number field $K$, we can take our sequence $a$ to be a nonlinear polynomial sequence in Theorem \ref{Application3Intro}(iii), which recovers a special case of Theorem \ref{Application2Intro}. We can also use sequences such as $a(n_1,\cdots,n_d) = (p_1(n_1),\cdots,p_d(n_d))$ with $p_i \in x^2\mathbb{Z}[x]$, which are polynomial sequences that do not arise from an integral domain structure on $\mathbb{Z}^d$. If we do not want to assume total ergodicity of our actions, we can use sequences such as $a(n_1,\cdots,n_d) = (\lfloor n_1^{t_1}\rfloor,\cdots,\lfloor n_d^{t_d}\rfloor)$ with $1 < t_i \in \mathbb{R}\setminus\mathbb{N}$. More generally, we can use $a(n_1,\cdots,n_d) = (\lfloor f_1(n_1)\rfloor,\cdots,\lfloor f_d(n_d)\rfloor)$ where $f_i:\mathbb{R}\rightarrow\mathbb{R}$ are tempered functions.\footnote{Since we will not be using tempered functions later on, we refer the reader to \cite{WMImpliesWMForTempered} for a definition and examples.} If we have $G = (\mathbb{Q},+)$ or $G = (\mathbb{Q}[i],+)$, then we can take $a$ to be a polynomial sequence in Theorem \ref{Application3Intro}(i)-(ii) to recover a special case of Theorem \ref{Application1Intro}. It is worth noting here that when $G$ is the additive group of a field of characteristic $0$, then there is no $\chi \in \widehat{G}$ that vanishes on a finite index subgroup of $G$, and this corresponds to the fact that any ergodic action of $G$ will automatically be totally ergodic.

Another interesting case to consider is $G = \mathbb{Z}^\infty := \oplus_{n = 1}^\infty\mathbb{Z}$. Similar to the case of $\mathbb{Z}^d$, the sequence $a:\mathbb{Z}^\infty\rightarrow\mathbb{Z}^\infty$ given by

\begin{equation}
    a(x_1,x_2,\cdots,x_n,\cdots) = (\lfloor f_1(x_1)\rfloor,\lfloor f_2(x_2)\rfloor,\cdots,\lfloor f_n(x_n)\rfloor,\cdots),
\end{equation}
where $f_i:\mathbb{R}\rightarrow\mathbb{R}$ are tempered functions with $f_i(0) = 0$, satisfies Theorem \ref{Application3Intro}(i)-(ii). However, the situation for polynomial sequences is more subtle, as Equation \eqref{UniformDistributionOfDifferencesAssumptionEquation} does not need to be satisfied for infinitely many $h \in G$ when $a:G\rightarrow G$ is a quadratic polynomial sequence and $\chi \in \widehat{G}$ satisfies $[G:\text{ker}(\chi)] = \infty$. An example of this phenomenon is witnessed by

\begin{equation}
    a\left(x_1,\cdots,x_n,\cdots\right) = \left(x_1^2,\cdots,x_n^2,\cdots\right), \chi(x_1,\cdots,x_n,\cdots) = \prod_{n = 1}^\infty \text{exp}\left(\frac{2\pi ix_n}{3}\right)\text{, and }h \in 3G.
\end{equation}
Since we normally expect polynomial sequences to have good averaging behavior with respect to $\chi \in \widehat{G}$ satisfying $[G:\text{ker}(\chi)] = \infty$, this example gives some credence to our assumption that $G$ should be finitely generated in Theorem \ref{Application3Intro}(iii). However, if we instead view $G$ as the additive group of the integral domain $\mathbb{Z}[t]$, then we see that polynomial sequence $a(y) \in \left(\mathbb{Z}[t]\right)[y]$ of degeree at least 2 satisfy Theorem \ref{Application3Intro}(iii) as a corollary of Theorem \ref{Application2Intro}.

Yet another interesting infinitely generated group to consider is $G = \oplus_{n = 1}^\infty\left(\mathbb{Z}/p\mathbb{Z}\right)$ with $p \in \mathbb{N}$ an odd prime. We observe that for each $\chi \in \widehat{G}$ we have $[G:\text{ker}(\chi)] = p$, so Equation \eqref{UniformDistributionOfDifferencesAssumptionEquation} is vacuously true for any $a:G\rightarrow G$ and any $\chi \in \widehat{G}$ that does not vanish on a subgroup of finite index. Furthermore, as in our previous examples, we can make interesting examples of sequences $a$ by imposing a ring structure on $G$. Firstly, if we let $\overline{\mathbb{F}_p}$ denote the unique countable field of characteristic $p$, then we have $(G,+) \cong (\overline{\mathbb{F}_p},+)$. We will show in Example 1 of Section \ref{DirectSumZmodp} that the sequence $a(g) = g^2$ satisfies Theorem \ref{Application3Intro}(i)-(ii). Alternatively, letting $\mathbb{F}_p$ denote the finite field of $p$ elements, we also have that $(G,+) \cong (\mathbb{F}_p[t],+)$, which imposes a different ring structure on $G$. We will also show in Example 2 of Section \ref{DirectSumZmodp} that when we consider an alternative notion of total ergodicity for $G$-actions that was introduced in \cite{AsymptoticTotalErgodicityOfFpActions}, then we can show that many polynomial sequences $a$ satisfy a version of Theorem \ref{Application3Intro}(iii).

\textbf{Acknowledgements:} I would like to thank Gerald Folland for help proving Lemma \ref{SpectralOrthogonalityLemma}. I would like to thank Ethan Ackelsberg for useful discussions about the group $\oplus_{n = 1}^\infty\left(\mathbb{Z}/p\mathbb{Z}\right)$. I would also like to thank the two anonymous referees for their careful reading of this article and their useful suggestions and corrections that greatly improved the quality of this article. I acknowledge being supported by grant
2019/34/E/ST1/00082 for the project “Set theoretic methods in dynamics and number theory,” NCN (The
National Science Centre of Poland).
\section{Central decompositions and the left regular representation}
The purpose of this section is to gather some facts about representations of groups so that we can state Theorem \ref{GeneralizedHahnHellinger} about central decompositions and the class of measures associated to a representation. This language will be needed to work with the left and right regular representations in the next section. The results of this section apply to any locally compact second countable Hausdorff topological group $G$ with left Haar measure $\lambda$. The reader is referred to \cite{C(star)AlgebraByDixmier} and \cite{ACourseInAbstractHarmonicAnalysisByFolland} for additional details.

All Hilbert spaces in this paper will be separable, and it is worth noting that the assumption that the probability space $(X,\mathscr{B},\mu)$ is standard implies that $L^2(X,\mu)$ is separable. A unitary representation $\pi$ of a group $G$ on a Hilbert space $\mathcal{H}$ is a strongly continuous homomorphism from $G$ to the space of unitary operators on $\mathcal{H}$, i.e., $\lim_{g\rightarrow e_G}||\pi(g)f-f|| = 0$ for all $f \in \mathcal{H}$ and $\pi(g)\pi(h) = \pi(gh)$ for all $g,h \in G$. If $\pi$ is a unitary representation of $G$ on $\mathcal{H}$ and $\rho$ is a unitary representation of $G$ on $\mathcal{H}'$, then a bounded linear map $T:\mathcal{H}\rightarrow\mathcal{H}'$ is called an \textit{intertwinning operator} if $T\pi(g) = \rho(g)T$ for all $g \in G$. The set of all intertwinning operators between $\pi$ and $\rho$ is denoted by $C(\pi,\rho)$. The representations $\pi$ and $\rho$ are \textit{unitarily equivalent} if $C(\pi,\rho)$ contains a unitary operator, and they are \textit{disjoint} if $C(\pi,\rho) = \{0\}$. The representations $\pi$ and $\rho$ are \textit{quasi-equivalent} if for any nontrivial subrepresentations $\pi'$ of $\pi$ and $\rho'$ of $\rho$, we have $C(\pi',\rho) \neq \{0\}$ and $C(\pi,\rho') \neq \{0\}$. The representation $\pi$ is \textit{irreducible} if there is no nontrivial closed subspace of $\mathcal{H}$ that is invariant under the action of $\pi$, and it is \textit{primary} if the center of $C(\pi) := C(\pi,\pi)$ consists of scalar multiples of the identity. We let $\widehat{G}$ denote the set of unitary equivalence classes of irreducible unitary representations of $G$, and observe that if $G$ is abelian then $\widehat{G}$ coincides with the Pontryagin dual. The space $\widehat{G}$ always admits a natural $\sigma$-algebra $\mathcal{M}$ of measurable sets known as the Mackey-Borel structure. We let \v{G} denote the set of quasi-equivalence classes of primary representations of $G$, and observe that \v{G} also admits a $\sigma$-algebra $\mathcal{M}'$ for which $(\widehat{G},\mathcal{M})$ can be viewed as a measurable subspace. It is worth mentioning that \textit{type I groups} are those groups $G$ for which $(\widehat{G},\mathcal{M})$ is a standard measure space, in which case it is also isomorphic to $(\text{\v{G}},\mathcal{M}')$.

Let us fix an infinite dimensional separable Hilbert space $\mathcal{H}_0$. If $(X,\mathscr{B},\mu)$ is a standard probability space, and for each $x \in X$ we have a Hilbert space $H_x \subseteq \mathcal{H}_0$, then we can create a new Hilbert space $\mathcal{H} := \int^{\oplus}H_xd\mu(x)$ that is the \textit{direct integral} of the spaces $H_x$. The elements $f \in\int^{\oplus}H_xd\mu(x)$ are measurable functions $f:X\rightarrow\mathcal{H}_0$ for which $f(x) \in H_x$ for all $x \in X$ and $\int_X||f(x)||_{H_x}^2d\mu(x) < \infty.$ For $f,g \in \int^{\oplus}H_xd\mu(x)$, we have $\langle f,g\rangle_{\mathcal{H}} := \int_X\langle f(x),g(x)\rangle_{H_x}d\mu(x)$. If $H_x = H$ for all $x \in X$, then we write $L^2(X,\mu,H)$ instead of $\int^{\oplus}Hd\mu(x)$. Now for each $1 \le n \le \infty$, consider $A_n := \{x \in X\ |\ \text{dim}(H_x) = n\}$. There exists a unitary isomorphism 

\begin{equation}\label{EquationForDirectIntegralIsomorphismToCountableDirectProduct}
    U:\int^{\oplus}H_xd\mu(x)\rightarrow\left(\bigoplus_{n = 1}^\infty L^2\left(A_n,\mu,\mathbb{C}^n\right)\right)\oplus L^2\left(A_\infty,\mu,\ell^2(\mathbb{N})\right),
\end{equation}
which verifies that $\int^{\oplus}H_xd\mu(x)$ is separable. Suppose that for each $x \in X$ we have a linear operator $T_x:H_x\rightarrow H_x$, and the collection $(T_x)_{x \in X}$ is measurable with respect to the $\sigma$-algebra of $X$. Assume there exists an $M < \infty$ such that for all $x \in X$ we have $||T_x||_{\mathcal{B}(H_x)} \le M$, where $\mathcal{B}(H_x)$ denotes the space of bounded linear operators from $H_x$ to itself and $||\cdot||_{\mathcal{B}(H_x)}$ denotes the norm on $\mathcal{B}(H_x)$. Then we can create a new operator

\begin{equation}
    \int^{\oplus}T_xd\mu(x):\int^{\oplus}H_xd\mu(x)\rightarrow\int^{\oplus}H_xd\mu(x)\text{ given by }\left(\int^{\oplus}T_xd\mu(x)\right)(f)(x) = T_x(f(x)),
\end{equation}
and it will have norm at most $M$. Similarly, if for $x \in X$ we have a measurable collection of representations $\pi_x$ of $G$ on $H_x$, then we can create a new representation $\pi := \int^{\oplus}\pi_xd\mu(x)$ of $G$ on $\int^{\oplus}H_xd\mu(x)$ given by

\begin{equation}
    \pi(g) = \int^{\oplus}\pi_x(g)d\mu(x).
\end{equation}
The next theorem tells us that any representation of $G$ can be constructed as a direct integral of primary representations.

\begin{theorem}[{\cite[Theorem 7.42]{ACourseInAbstractHarmonicAnalysisByFolland}}]\label{GeneralizedHahnHellinger}
    Let $\pi$ be a unitary representation of $G$ on a Hilbert space $\mathcal{H}$. Let $\{\pi_p\}_{p \in \text{\v{G}}}$ be a measurable collection of primary representations of $G$ with each $\pi_p$ belonging to the quais-equivalence class of $p$ and acting on $H_p$. There exist a standard measure $\mu$ on \v{G} whose measure class $[\mu]$ is uniquely determined, so that $\pi$ is unitarily equivalent to $\rho := \int^{\oplus}\pi_pd\mu(p)$.
\end{theorem}
The decomposition $\pi \cong \int^{\oplus}\pi_pd\mu(p)$ is the \textit{central decomposition} of $\pi$ and the measure class $[\mu_{\pi}]$ of $\mu$ is the \textit{class of measures associated with $\pi$}. For $f \in \mathcal{H}$, we let $\mathcal{H}_f := \overline{\text{Span}_{\mathbb{C}}\{\pi(g)f\ |\ g \in G\}}$, $\pi_f = \pi|_{\mathcal{H}_f}$, and define $[\mu_{\pi_f}]$ to be the spectral type of $f$ (with respect to $\pi$). The next lemma verifies that when $G$ is abelian, the previous notion of spectral type coincides with that introduced in Section 1.2.

\begin{lemma}
    Let $G$ be abelian, $(X,\mathscr{B},\nu,(T_g)_{g \in G})$ a measure preserving $G$-system, and let $\pi(g) = T_g$ be the unitary representation of $G$ on $L^2(X,\nu)$ induced by $T$. For any $f \in L^2(X,\nu)$ we have $\nu_{T,f} \in [\mu_{\pi_f}]$.
\end{lemma}

\begin{proof}
    Since $G$ is abelian, \v{G} is simply the Pontryagin dual $\widehat{G}$, so when we use Theorem \ref{GeneralizedHahnHellinger} we may assume without loss of generality that $\pi$ is acting on $L^2(\widehat{G},\mu)$ by $(\pi(g)F)(\chi) = \chi(g)F(\chi)$ for $\chi \in \widehat{G}$ and $F \in L^2(\widehat{G},\mu)$. Letting $f'$ denote the image of $f$ in $L^2(\widehat{G},\mu)$, we may take $\mu_{\pi_f} = |f'|^2d\mu$. We now see that

    \begin{alignat*}{2}
        &\hat{\mu}_{\pi_f}(g) = \int_{\widehat{G}}\chi(g)d\mu_{\pi_f}(\chi) = \int_{\widehat{G}}\chi(g)|f'|^2d\mu(\chi) = \langle \pi(g)f',f'\rangle_{L^2(\widehat{G},\mu)} = \langle T_gf,f\rangle_{L^2(X,\nu)} = \hat{\nu}_{T,f}(g).
    \end{alignat*}
    Since a finite measure on $\widehat{G}$ is determined by its Fourier coefficients, we see that $\mu_{\pi_f} = \nu_{T,f}$.
\end{proof}

\begin{lemma}\label{SpectralOrthogonalityLemma}
    Let $\pi$ be a unitary representation of $G$ on $\mathcal{H}$. If $f_1,f_2 \in \mathcal{H}$ are such that $[\mu_{\pi_{f_1}}]\perp[\mu_{\pi_{f_2}}]$, then $\mathcal{H}_{f_1}\perp\mathcal{H}_{f_2}$. In particular, we have $\langle f_1,f_2\rangle = 0$.
\end{lemma}

\begin{proof}
    Proposition 8.4.7 of \cite{C(star)AlgebraByDixmier} tells us that $\pi_{f_1}$ and $\pi_{f_2}$ are disjoint representations. Let us assume for the sake of contradiction that $\mathcal{H}_{f_1}\not\perp\mathcal{H}_{f_2}$ and let $P:\mathcal{H}_{f_1}\rightarrow\mathcal{H}_{f_2}$ denote the unique orthogonal projection. Since $\mathcal{H}_{f_1}$ and $\mathcal{H}_{f_2}$ are $\pi$-invariant, we see that for any $g \in G$, the map $P_g := \pi(g)P\pi(g)^{-1}:\mathcal{H}_{f_1}\rightarrow\mathcal{H}_{f_2}$ is also an orthogonal projection. By the uniqueness of $P$, we see that $\pi_{f_2}(g)P = \pi(g)P = P\pi(g) = P\pi_{f_1}(g)$, so $P$ is an intertwinning between $\pi_{f_1}$ and $\pi_{f_2}$, which contradicts the disjointness of $\pi_{f_1}$ and $\pi_{f_2}$.
\end{proof}
For the space $L^2(G,\lambda)$, the left regular representation $L$ of $G$ is given by $(L(g)f)(x) = f(g^{-1}x)$. Letting $\tilde{\lambda}$ be the measure given by $\tilde{\lambda}(A) = \lambda(A^{-1})$, we see that $\tilde{\lambda}$ is a right Haar measure. The right regular representation $R$ of $G$ is given by $(R(g)f)(x) = f(xg)$. Since $\Phi:L^2(G,\lambda)\rightarrow L^2(G,\tilde{\lambda})$ given by $(\Phi f)(x) = f(x^{-1})$ satisfies $L = \Phi^{-1}R\Phi$, we see that the left and right regular representations are equivalent, hence $[\mu_L] = [\mu_R]$. Due to the notation used in our applications to measure preserving systems, we will find it more convenient to work with the right regular representation in many situations.

In the case that $G$ is unimodular, we have $\lambda = \tilde{\lambda}$, so we may study the left and right regular representations simultaneously on the space $L^2(G,\lambda)$. When the group $G$ is assumed to be unimodular and type I the measure $\mu_L$ is the \textit{Plancherel measure}, which coincides with the Haar measure of $\widehat{G}$ when $G$ is abelian. Trying to study $\mu_L$ and related topics leads to the study of square integrable representation. From this theory we only require the following lemma, which shows why we have decided to use the assumptions present in Equations \eqref{AssumptionEquationForZvdC} and \eqref{AssumptionEquationForAbelianvdC}.

\begin{lemma}[{\cite[Lemma 14.1.1]{C(star)AlgebraByDixmier}}]\label{PositiveDefiniteLemma}
    Let $G$ be a unimodular group. If $\phi \in L^2(G,\lambda)$ is continuous and positive definite, then there exists $f \in L^2(G,\lambda)$ for which $\phi(g) = \langle L_gf,f\rangle$ for all $g \in G$.
\end{lemma}

\section{A Hilbert space of sequences}
In this section $\mathcal{H}$ and $\mathscr{H}$ will always denote Hilbert spaces and $G$ a countably infinite amenable group. Let $||\cdot||$ and $\langle\cdot,\cdot\rangle$ denote the norm and inner product on $\mathcal{H}$ and let $||\cdot||_{\mathscr{H}}$ and $\langle\cdot,\cdot\rangle_{\mathscr{H}}$ denote the norm and inner product on $\mathscr{H}$. We now want to construct a Hilbert space $\mathscr{H}$ whose vectors are sequences of vectors from $\mathcal{H}$, as this will be the natural setting in which to prove our results.

Given a right-F\o lner sequence $(F_n)_{n = 1}^\infty$ in $G$, denote the collection of square averageable sequences by
 
 \begin{equation}
     SA(\mathcal{H},(F_n)_{n = 1}^\infty) := \left\{f:G\rightarrow\mathcal{H}\ |\ \limsup_{N\rightarrow\infty}\frac{1}{|F_N|}\sum_{g \in F_N}||f(g)||^2\ < \infty\right\}.
 \end{equation}
 Furthermore, we denote the collection of uniformly bounded sequences by

 \begin{alignat*}{2}
     UB(\mathcal{H}) := \bigg\{f:G\rightarrow\mathcal{H}\ |\ \sup_{g \in G}||f(g)|| < \infty\bigg\}.
 \end{alignat*}
Let $f_1,f_2 \in SA(\mathcal{H},(F_n)_{n = 1}^\infty)$ and observe that

\begin{alignat*}{2}
    &\limsup_{N\rightarrow\infty}\frac{1}{|F_N|}\left|\sum_{g \in F_N}\langle f_1(g), f_2(g)\rangle\right| \le \limsup_{N\rightarrow\infty}\frac{1}{|F_N|}\sum_{g \in F_N}||f_1(g)||\cdot||f_2(g)||\\
    \le & \left(\limsup_{N\rightarrow\infty}\frac{1}{|F_N|}\sum_{g \in F_N}||f_1(g)||^2\right)^{\frac{1}{2}}\left(\limsup_{N\rightarrow\infty}\frac{1}{|F_N|}\sum_{g \in F_N}||f_2(g)||^2\right)^{\frac{1}{2}} < \infty.
\end{alignat*}
It follows that given any countable set $I$ and $\{f_n\}_{n \in I} \subseteq SA(\mathcal{H},(F_n)_{n = 1}^\infty)$ we may use diagonalization to construct an increasing sequence of positive integers $(N_q)_{q = 1}^{\infty}$ for which

\begin{equation}
\lim_{q\rightarrow\infty}\frac{1}{|F_{N_q}|}\sum_{g \in F_{N_q}}\langle x(g),y(g)\rangle
\end{equation}
exists whenever $x,y \in \{f_n\}_{n \in I}$ and $h \in G$. We now construct a new Hilbert space $\mathscr{H} = \mathscr{H}\left(\{f_n\}_{n \in I},(F_{N_q})_{q = 1}^{\infty}\right)$ from $\{f_n\}_{n \in I}$ and $(N_q)_{q = 1}^{\infty}$ as follows. For $x,y \in \left\{f_n\right\}_{n \in I}$, we define

\begin{equation}
\left\langle x,y\right\rangle_{\mathscr{H}} := \lim_{q\rightarrow\infty}\frac{1}{|F_{N_q}|}\sum_{g \in F_{N_q}}\langle x(g),y(g)\rangle,
\end{equation}
so $\langle\cdot,\cdot\rangle_{\mathscr{H}}$ defines a sesquilinear form on $\mathscr{H}' := \text{Span}_{\mathbb{C}}\left(\left\{f_n\ |\ n \in I\right\}\right)$ with scalar multiplication and addition occuring pointwise. Letting 

\begin{alignat*}{2}
    \mathscr{H}'' := \Biggl\{E \in SA(\mathcal{H}, (F_{N_q})_{q = 1}^\infty)\ |\ &\forall\ \epsilon > 0\ \exists\ E_\epsilon \in \mathscr{H}'\text{ s.t.}\\ &\limsup_{q\rightarrow\infty}\frac{1}{|F_{N_q}|}\sum_{g \in F_{N_q}}||E(g)-E_\epsilon(g)||^2 < \epsilon\Biggl\}\text{, and}
\end{alignat*}

\begin{equation}
S = \left\{x \in \mathscr{H}''\ |\ \lim_{q\rightarrow\infty}\frac{1}{|F_{N_q}|}\sum_{g \in F_{N_q}}||x(g)||^2 = 0\right\},
\end{equation}
we see that $\mathscr{H}''/S$ is a pre-Hilbert space. We will soon see that $\mathscr{H}''$ is sequentially closed under the topology induced by $\langle\cdot,\cdot\rangle_{\mathscr{H}}$, so we define $\mathscr{H}(\{f_n\}_{n = 1}^\infty,\allowbreak(F_{N_q})_{q = 1}^{\infty}) = \mathscr{H}''/S$. We call $\mathscr{H}(\{f_n\}_{n = 1}^\infty,\allowbreak(F_{N_q})_{q = 1}^{\infty})$ the Hilbert space induced by $(\{f_n\}_{n = 1}^\infty,\allowbreak(F_{N_q})_{q = 1}^{\infty})$. We may write $\mathscr{H}$ in place of $\mathscr{H}(\{f_n\}_{n = 1}^\infty,\allowbreak(F_{N_q})_{q = 1}^{\infty})$ if $(\{f_n\}_{n = 1}^\infty,\allowbreak(F_{N_q})_{q = 1}^{\infty})$ is understood from the context.\\

For a countable set $I$ and $\{f_n\}_{n \in I} \subseteq SA(\mathcal{H},(F_n)_{n = 1}^\infty)$ and $(N_q)_{q = 1}^\infty \subseteq \mathbb{N}$ we say that $(\{f_n\}_{n \in I},\allowbreak(F_{N_q})_{q = 1}^{\infty})$ is a \textbf{permissible pair} if $\mathscr{H} := \mathscr{H}(\{f_n\}_{n \in I},\allowbreak(F_{N_q})_{q = 1}^{\infty})$ is well defined, i.e., if all limits used in the construction of $\mathscr{H}$ exist. Given a $x:G\rightarrow\mathcal{H}$ for which $x \in \mathscr{H}''$, we may view $x$ as an element of $\mathscr{H}$ by identifying $x$ with its equivalence class in $\mathscr{H}''/S$. We will now show that $\mathscr{H}$ is a Hilbert space by verifying that it is complete.

\begin{theorem}
\label{CompletenessOfMathscrH}
Let $\mathcal{H}$ be a Hilbert space, let $I$ be a countable set, and let $\{f_n\}_{n \in I} \subseteq SA(\mathcal{H},(F_n)_{n = 1}^\infty)$. Let $ P = \left(\{f_n\}_{n \in I},\allowbreak(F_{N_q})_{q = 1}^{\infty}\right)$ be a permissible pair and $\mathscr{H} = \mathscr{H}(P)$. If $\allowbreak\left\{\xi_m\right\}_{m = 1}^{\infty} \subseteq \mathscr{H}''$ is a Cauchy sequence with respect to the metric induced by $||\cdot||_{\mathscr{H}}$, then there exists $\xi \in \mathscr{H}''$ for which

\begin{equation}
    \lim_{m\rightarrow\infty}\left(\lim_{q\rightarrow\infty}\frac{1}{|F_{N_q}|}\sum_{g \in F_{N_q}}||\xi_m(g)-\xi(g)||^2\right) = 0.
\end{equation}
In particular, $\mathscr{H}$ is a Hilbert space.
\end{theorem}

\begin{proof} We proceed by modifying the proof of the main result in section \textsection 2 of chapter II of \cite{BohrAndFolner}. Let $(\epsilon_m)_{m = 1}^{\infty}$ be a sequence of real numbers tending to 0 for which

\begin{equation}
    \lim_{q\rightarrow\infty}\frac{1}{|F_{N_q}|}\sum_{g \in F_{N_q}}||\xi_m(g)-\xi_k(g)||^2 < \epsilon_m
\end{equation}
whenever $k \ge m$. By induction, let $T_0 = N_0 = 0$ and let $(T_m)_{m = 1}^{\infty} \subseteq \mathbb{N}$ be increasing and satisfy conditions (i)-(iii) below.

\begin{itemize}
\item[(i)] For every $m \ge 1$, every $k \ge m$, and every $T \ge T_k$ 

\begin{equation}
    \frac{1}{|F_{N_T}|}\sum_{g \in F_{N_T}}||\xi_k(g)-\xi_m(g)||^2 < \epsilon_m.
\end{equation}

\item[(ii)] For every $m \ge 1$ and every $k \ge m$

\begin{equation}
    \frac{1}{|F_{N_{T_k}}|-|F_{N_{T_{k-1}}}|}\sum_{g \in F_{N_{T_k}}\setminus F_{N_{T_{k-1}}}}||\xi_k(g)-\xi_m(g)||^2 < \epsilon_m.
\end{equation}

\item[(iii)] For every $m \ge 1$

\begin{equation}
    \frac{1}{|F_{N_{T_m}}|}\sum_{j = 1}^{m-1}\sum_{F_{N_{T_j}}\setminus F_{N_{T_{j-1}}}}||\xi_j(g)-\xi_m(g)||^2 < \epsilon_m.
\end{equation}
\end{itemize}

Now let us define $\xi$ by $\xi(g) = \xi_m(g)$ where $m$ the least integer for which $g \in F_{N_{T_m}}$ if such an $m$ exists, and $\xi(g) = 0$ if no such $m$ exists. To conclude the proof, we let $F_0 = \emptyset$ and note that for $m \ge 1$, $k > m$ and $T_{k-1} < T \le T_k$ we have

\begin{alignat*}{2}
    & \sum_{g \in F_{N_T}}||\xi_m(g)-\xi(g)||^2\\
    \le & \sum_{j = 1}^{m-1}\sum_{g \in F_{N_{T_j}}\setminus F_{N_{T_{j-1}}}}||\xi_j(g)-\xi_m(g)||^2+\sum_{j = m}^{k-1}\sum_{g \in F_{N_{T_j}}\setminus F_{N_{T_{j-1}}}}||\xi_m(g)-\xi(g)||^2\\
    &+\sum_{g \in F_{N_T}\setminus F_{N_{T_{k-1}}}}||\xi_m(g)-\xi(g)||^2\\
    \le & |F_{N_{T_m}}|\epsilon_m+\sum_{j = m}^{k-1}(|F_{N_{T_j}}|-|F_{N_{T_{j-1}}}|)\epsilon_m+\sum_{g \in F_{N_T}}||\xi_k(g)-\xi_m(g)||^2\\
    \le & (|F_{N_{T_m}}|+|F_{N_{T_{k-1}}}|-|F_{N_{T_{m-1}}}|+|F_{N_T}|)\epsilon_m \le 3|F_{N_T}|\epsilon_m.
\end{alignat*}
\end{proof}


For $h \in G$ and $f:G\rightarrow\mathcal{H}$, define $(U_hf)(g) = f(gh)$, and we observe that for $f \in \text{SA}(\mathcal{H},(F_n)_{n = 1}^\infty)$, it is possible that $U_hf \notin \text{SA}(\mathcal{H},(F_n)_{n = 1}^\infty)$ for some $h \in G$. To see an example of this, let $G = \mathbb{Z}$, let $F_n = [n^2,n^2+n]$, let $\mathcal{H} = \mathbb{C}$, and let $f(m) = 1$ for $m \in [n^2,n^2+n]$, let $f(m) = n^3$ for $m \in (n^2+n,(n+1)^2)$, and let $f(m) = 0$ for all other values of $m$.

However, if $\{f_n\}_{n \in I} \subseteq \text{UB}(\mathcal{H})$, then $\{U_gf_n\}_{n \in I, g \in G} \subseteq \text{UB}(\mathcal{H})$. Furthermore, for any permissible pair $P := (\{U_gf_n\}_{n \in I, g \in G},(F_N)_{N = 1}^\infty)$ and $\mathscr{H} = \mathscr{H}(P)$, we have $||f_n||_{\mathscr{H}} = ||U_gf_n||_{\mathscr{H}}$ for all $n \in I$ and $g \in G$ . Consequently, we may define $U_g$ on $\mathscr{H}$ by defining the action of $U_g$ on the equivalence class of $f_n$ to be the equivalence class of $U_gf_n$. The maps $U_g$ then have unique extensions to all of $\mathscr{H}$, so $U$ will be a unitary representation of $G$ on $\mathscr{H}$. Our next definitions are conditions on $f \in UB(\mathcal{H},(F_n)_{n = 1}^\infty)$ that ensure that the unitary action of $G$ on the equivalence class of $f$ in $\mathscr{H}$ is sufficiently mixing.

\begin{definition}
    Let $G$ be an amenable group. Let $(F_N)_{N \in \mathbb{N}}$ be a right-F\o lner sequence in $G$, $\mathcal{H}$ a Hilbert space, and $u \in \text{UB}(\mathcal{H})$. The function $u$ is \textbf{nearly weakly mixing} if for any $f \in UB(\mathcal{H},(F_n)_{n = 1}^\infty)$ and any permissible pair $P = (\{U_gu,U_hf\}_{g,h \in G},(F_{N_q})_{q = 1}^\infty)$ we have

    \begin{equation}
        \lim_{r\rightarrow\infty}\frac{1}{|F_N|}\sum_{h \in F_N}\left|\lim_{q\rightarrow\infty}\frac{1}{|F_{N_q}|}\sum_{g \in F_{N_q}}\langle u(gh), f(g)\rangle\right| = 0.
    \end{equation}
    The function $u$ is \textbf{spectrally Lebesgue} if for any permissible pair $P = (\{U_gu\}_{g \in G},(F_{N_q})_{q = 1}^\infty)$ we have that the unitary representation $U$ of $G$ on $\mathscr{H} = \mathscr{H}(P)$ is equivalent to a subrepresentation of the left regular representation of $G$.
\end{definition}

Our next result is the promised analogue of Theorem \ref{ClassicalvanderCorput'sDifferenceTheorems}(iii) for amenable groups that corresponds to weak mixing, which will also be seen to be a generalization of \cite[Theorem 2.12]{vanderCorputTheoremSurvey} after we prove Theorem \ref{StrongDisjointnessOfSequences}.

\begin{theorem}\label{WeakMixingAmenablevdC}
    Let $G$ be a countably infinite
amenable group. Let $(F_N)_{N = 1}^\infty$ be a right-F\o lner sequence in $G$. Let $\mathcal{H}$ be a Hilbert space and $u \in \text{UB}(\mathcal{H})$. If

\begin{equation}\label{AmenableWeakMixingVDCCondition}
    \lim_{N\rightarrow\infty}\frac{1}{|F_N|}\sum_{h \in F_N}\limsup_{N\rightarrow\infty}\left|\frac{1}{|F_N|}\sum_{g \in F_N}\langle u(gh),u(g)\rangle \right| = 0,
\end{equation}
then $u$ is a nearly weakly mixing sequence.
\end{theorem}

\begin{proof}
    Let $f \in UB(\mathcal{H},(F_N)_{N = 1}^\infty)$ be arbitrary and let $(N_q)_{q = 1}^\infty$ be such that $(\{U_gu,U_hf\}_{g,h \in G},(F_{N_q})_{q = 1}^\infty)$ is a permissible pair, which induces the Hilbert space $\mathscr{H}$. If $||u||_{\mathscr{H}} = 0$, then it is clear that $u$ is a nearly weakly mixing sequence, so let us assume that this is not the case. Let $\mathscr{H}_u = \overline{\text{Span}_{\mathbb{C}}\{U_gu\}_{g \in G}}$, and let $P:\mathscr{H}\rightarrow\mathscr{H}_u$ denote the orthogonal projection. Let $\epsilon > 0$ be arbitrary and let $(c_g)_{g \in G}$ be a finitely supported sequence of complex numbers for which

    \begin{equation}
        \left|\left|Pf-\sum_{g \in G}c_gU_gu\right|\right| < \frac{\epsilon}{||u||_{\mathscr{H}}},
    \end{equation}
    and observe that

    \begin{alignat*}{2}
        &\lim_{N\rightarrow\infty}\frac{1}{|F_N|}\sum_{h \in F_N}\left|\lim_{q\rightarrow\infty}\frac{1}{|F_{N_q}|}\sum_{g \in F_{N_q}}\langle u(gh), f(g)\rangle\right|\\
        = & \lim_{N\rightarrow\infty}\frac{1}{|F_N|}\sum_{h \in F_N}\left|\langle U_hu, f\rangle_{\mathscr{H}}\right| = \lim_{N\rightarrow\infty}\frac{1}{|F_N|}\sum_{h \in F_N}\left|\langle U_hu, Pf\rangle_{\mathscr{H}}\right|\\
        \le & \sum_{g \in G}\lim_{N\rightarrow\infty}\frac{1}{|F_N|}\sum_{h \in F_N}\left|\langle U_hu, c_gU_gu\rangle_{\mathscr{H}}\right|+\epsilon\\
        = &\sum_{g \in G}|c_g|\lim_{N\rightarrow\infty}\frac{1}{|F_N|}\sum_{h \in F_N}\left|\langle U_{g^{-1}h}u, u\rangle_{\mathscr{H}}\right|+\epsilon\\
        = & \sum_{g \in G}|c_g|\lim_{N\rightarrow\infty}\frac{1}{|F_N|}\sum_{h \in F_N}\left|\langle U_{h}u, u\rangle_{\mathscr{H}}\right|d\lambda(h)+\epsilon = \epsilon.
    \end{alignat*}
\end{proof}

\begin{theorem}\label{LebesgueSpectrumAmenablevdC}
    Let $G$ be a countably infinite
amenable group and $(F_N)_{N \in \mathbb{N}}$ a F\o lner sequence in $G$. Let $\mathcal{H}$ be a Hilbert space and $u \in \text{UB}(\mathcal{H})$. If

\begin{equation}\label{AmenableLebesgueSpectrumVDCCondition}
    \sum_{g \in G}\limsup_{N\rightarrow\infty}\left|\frac{1}{|F_N|}\sum_{g \in F_N}\langle u(gh),u(g)\rangle \right|^2 < \infty,
\end{equation}
then $u$ is spectrally Lebesgue.
\end{theorem}

\begin{proof}
    Let $P := (\{U_gu\}_{g \in G},(F_{N_q})_{q = 1}^\infty)$ be a permissible pair, which induces the Hilbert space $\mathscr{H} = \mathscr{H}(P)$. Let $\phi:G\rightarrow\mathbb{C}$ be given by

    \begin{equation}
        \phi(h) = \lim_{q\rightarrow\infty}\frac{1}{|F_{N_q}|}\sum_{g \in F_{N_q}}\langle u(gh), u(g)\rangle = \langle U_hu,u\rangle_{\mathscr{H}}.
    \end{equation}
    Since $\phi \in L^2(G,\lambda)$ is positive definite and continuous, we may use Lemma \ref{PositiveDefiniteLemma} to pick a $f \in L^2(G,\lambda)$ for which $\phi(g) = \langle L_gf,f\rangle$. Letting $H' = \overline{\{L_gf\ |\ g \in G\}}$, we see that there is a unique isomorphism $\Phi:\mathscr{H}\rightarrow H'$ determined by $\Phi(U_gu) = L_gf$, which yields the desired result.
\end{proof}

\begin{definition}
    Let $G$ be an
amenable group and $(F_N)_{N \in \mathbb{N}}$ a F\o lner sequence in $G$. Let $\mathcal{H}$ be a Hilbert space and $w \in \text{UB}(\mathcal{H})$. The function $w$ is \textbf{compact} if for any permissible pair $P = (\{U_gw\}_{g \in G},(N_q)_{q = 1}^\infty)$ and $\mathscr{H} = \mathscr{H}(P)$, the representation $U_w$ decomposes into a direct sum of irreducible finite dimensional representations. The function $w$ is \textbf{spectrally singular} if for any permissible pair $P = (\{U_gw\}_{g \in G},(N_q)_{q = 1}^\infty)$ and $\mathscr{H} = \mathscr{H}(P)$, we have $[\mu_{U_w}]\perp[\mu_L]$.
\end{definition}

It is worth observing that $w$ is compact if and only if $[\mu_{U_w}]$ is a discrete measure supported on the quasi-equivalence classes of finite dimensional irreducible representations of $G$. Furthermore, $w$ is nearly weakly mixing if and only if $[\mu_{U_w}]$ is a measure with no atoms on the quasi-equivalence classes of finite dimensional irreducible representations. With this in mind, the following classical decomposition for unitary actions of an amenable group can be seen as a variation of Lemma \ref{SpectralOrthogonalityLemma}. 

\begin{theorem}[{\cite[Lemma 3.4]{CompactWeakMixingForAmenableActions}}]\label{CompactWeakMixingDecomposition}
    Let $G$ be a countably infinite amenable group and $(F_n)_{n = 1}^\infty$ a right-F\o lner sequence.\footnote{A posteriori we see that the decomposition is independent of the choice of F\o lner sequence $(F_n)_{n = 1}^\infty$.} If $\pi$ is a representation of $G$ on $\mathcal{H}$, then we have the orthogonal decomposition $\mathcal{H} = \mathcal{H}_c\oplus\mathcal{H}_w$ with $\mathcal{H}_c$ decomposing into a direct sum of irreducible finite dimensional representations of $G$ and

    \begin{equation}
        \mathcal{H}_w := \left\{x \in \mathcal{H}\ |\ \lim_{N\rightarrow\infty}\frac{1}{|F_N|}\sum_{g \in F_N}|\langle \pi(g)x,x\rangle| = 0\right\}.\footnote{It is worth mentioning that in \cite{CompactWeakMixingForAmenableActions}, $\mathcal{H}_w := \{x \in \mathcal{H}\ |\ M(\langle \pi(g)x,x\rangle) = 0$\}, where $M$ is the unique invariant mean on the space of weakly almost periodic functions on $G$. This definition coincides with ours.}
    \end{equation}
\end{theorem}

\begin{lemma}\label{WeakDisjointnessOfSequences}
    Let $G$ be a countably infinite
amenable group and let $(F_N)_{N \in \mathbb{N}}$ be a right-F\o lner sequence in $G$. Let $\mathcal{H}$ be a Hilbert space and $u,w \in \text{UB}(\mathcal{H})$.
\begin{enumerate}[(i)]
    \item If $u$ is spectrally Lebesgue and $w$ is spectrally singular, then

\begin{equation}\label{WeaklyOrthogonalEquation}
    \lim_{n\rightarrow\infty}\frac{1}{|F_n|}\sum_{g \in F_n}\langle u(g),w(g)\rangle = 0.
\end{equation}

    \item If $u$ is nearly weakly mixing and $w$ is compact, then we again have Equation \eqref{WeaklyOrthogonalEquation}.
\end{enumerate}
\end{lemma}

\begin{proof}
    Let $(N_q)_{q = 1}^\infty \subseteq \mathbb{N}$ be any sequence for which

    \begin{equation}
        \lim_{q\rightarrow\infty}\frac{1}{|F_{N_q}|}\sum_{g \in F_{N_q}}\langle u(g),w(g)\rangle
    \end{equation}
    exists. By passing to a subsequence if necessary, we may assume without loss of generality that $P = (\{U_gu,U_hw\}_{g,h \in G},(F_{N_q})_{q = 1}^\infty)$ is a permissible pair, and we may set $\mathscr{H} = \mathscr{H}(P)$. In (i) we see that $[\mu_{U_u}]\perp[\mu_{U_w}]$, so Lemma \ref{SpectralOrthogonalityLemma} tells us that

    \begin{equation}
        0 = \langle u,w\rangle_{\mathscr{H}} = \lim_{q\rightarrow\infty}\frac{1}{|F_{N_q}|}\sum_{g \in F_{N_q}}\langle u(g),w(g)\rangle.
    \end{equation}
    To reach the same conclusion in (ii), we use Theorem \ref{CompactWeakMixingDecomposition}. The desired result now follows from the fact that $(N_q)_{q = 1}^\infty$ was arbitrary.
\end{proof}

\begin{theorem}\label{StrongDisjointnessOfSequences}
    Let $G$ be a countably infinite
amenable group and let $(F_N)_{N \in \mathbb{N}}$ be a right-F\o lner sequence in $G$. Let $(X,\mathscr{B},\nu)$ be a $\sigma$-finite measure space, let $u,w:G\rightarrow L^\infty(X,\mu)$ be such that $u,w \in \text{UB}(L^2(X,\nu))$, $\text{sup}_{g \in G}||w(g)||_{\infty} < \infty$, and $\text{sup}_{g \in G}||u(g)||_{\infty} < \infty$. Let $c:G\rightarrow\mathbb{C}$ be bounded.
\begin{enumerate}[(i)]
    \item If $u$ is spectrally Lebesgue, and $w$ is spectrally singular, then

\begin{equation}\label{StronglyOrthogonalEquation}
    \lim_{n\rightarrow\infty}\left|\left|\frac{1}{|F_n|}\sum_{g \in F_n} u(g)w(g)\right|\right|_2 = 0.
\end{equation}

\item If $u$ is spectrally Lebesgue, and $c$ is spectrally singular, then

\begin{equation}\label{StronglyOrthogonalEquation2}
    \lim_{n\rightarrow\infty}\left|\left|\frac{1}{|F_n|}\sum_{g \in F_n} u(g)c(g)\right|\right|_2 = 0.
\end{equation}

    \item If $u$ is nearly weakly mixing and $w$ is compact, then we again have Equation \eqref{StronglyOrthogonalEquation}. If $u$ is nearly weakly mixing and $c$ is compact, then we again have Equation \eqref{StronglyOrthogonalEquation2}.
\end{enumerate}
\end{theorem}

\begin{proof}
    We prove the stated results for $w$ and remark that the proof for $c$ is identical. We assume without loss of generality that $\text{sup}_{g \in G}||w(g)||_{\infty} \le 1$, and $\text{sup}_{g \in G}||u(g)||_{\infty} \le 1$. We also observe that $\overline{w}$ is spectrally singular (compact) if and only if $w$ is spectrally singular (compact). Let us assume for the sake of contradiction that there exists a sequence $(N_q)_{q = 1}^\infty \subseteq \mathbb{N}$ for which

    \begin{equation*}
    \epsilon := \lim_{q\rightarrow\infty}\left|\left|\frac{1}{|F_{N_q}|}\sum_{g \in F_{N_q}} u(g)w(g) \right|\right|_2 > 0\text{ and }\lim_{q\rightarrow\infty}\frac{|F_{N_{q-1}}|}{|F_{N_q}|} = 0.
\end{equation*}
Consider $(\xi_q')_{q = 1}^\infty \subseteq L^2(X,\mu)$ and $\xi = (\xi_g)_{g \in G} \in \mathscr{H}$ be given by

\begin{equation}
    \xi_q' = \frac{1}{|F_{N_q}\setminus F_{N_{q-1}}|}\sum_{g \in F_{N_q}\setminus F_{N_{q-1}}}u(g)w(g), \xi_g = \frac{\xi_q'}{||\xi_q'||_2}\text{ for }g \in F_{N_q}\setminus F_{N_{q-1}},
\end{equation}
and $\xi_g = 1$ for $g \notin\bigcup_{q = 1}^\infty F_{N_q}$. Since $||\xi_q'||_2 > \frac{\epsilon}{2}$ for all $q \ge q_0$, we may replace $(F_{N_q})_{q = 1}^\infty$ with $(F_{N_{q-q_0}})_{q = q_0+1}^\infty$ and assume without loss of generality that $||\xi_q'||_2 > \frac{\epsilon}{2}$ for all $q$ in order to ensure that $\xi_g$ is well defined and satisfies $\sup_{g \in G}||\xi_g||_\infty \le \frac{2}{\epsilon}$. We see that $\xi$ is invariant under the action of $G$ in the sense that for all $h \in G$ we have

\begin{alignat*}{2}
    &||U_h\xi-\xi||_{\mathscr{H}} =  \lim_{q\rightarrow\infty}\frac{1}{|F_{N_q}|}\sum_{g \in F_{N_q}}||\xi_{gh}-\xi_g||_2\\
    = & \lim_{q\rightarrow\infty}\frac{1}{|F_{N_q}\setminus F_{N_{q-1}}|}\sum_{g \in F_{N_q}\setminus F_{N_{q-1}}}\left|\left|\frac{\xi_q'}{||\xi_q'||_2}-\frac{\xi_q'}{||\xi_q'||_2}\right|\right|_2 = 0.
\end{alignat*}
Let $(N_q')_{q = 1}^\infty$ be a subsequence of $(N_q)_{q = 1}^\infty$ for which we may let $\mathscr{H} = \mathscr{H}(\{(U_h\xi U_g\overline{w}\}_{g,h \in G},(N_q')_{q = 1}^\infty)$ and let $\mathscr{H}_w := \mathscr{H}(\{U_g\overline{w}\}_{g \in G},(N_q')_{q = 1}^\infty)$. Let $M_\xi:\mathscr{H}_w\rightarrow\mathscr{H}$ denote the multiplication by $\xi$ map. We see that $M_\xi$ is a bounded linear map since $\sup_{g \in G}||\xi_g||_\infty \le 1$, and if we were working with $c$ instead of $w$ then $M_\xi$ would still be a bounded linear map since $\sup_{g \in G}||\xi_g||_2 = 1$. We see that for any $g \in G$ and $w' \in \mathscr{H}_w$ of the form $w' = U_h\overline{w}$ we have

\begin{alignat*}{2}
    &||U_gM_\xi w'-M_\xi U_gw'||_{\mathscr{H}} = ||U_g(\xi w')-\xi U_gw'||_{\mathscr{H}} = ||U_g\xi U_gw'-\xi U_gw'||_{\mathscr{H}} \le ||U_g\xi-\xi||_{\mathscr{H}} = 0.
\end{alignat*}
Since the span of such $w'$ is dense in $\mathscr{H}_w$, we see that $U_gM_\xi w' = M_\xi U_hw'$ for all $w' \in \mathscr{H}_w$, so $M_\xi$ is an intertwinning operator for the representation $U$ of $G$ on $\mathscr{H}_w$ and the representation $U$ of $G$ on $\mathscr{H}$. It follows that $U_{\xi \overline{w}}$ is quasi-equivalent to a subrepresentation of $U_{\overline{w}}$, so $[\mu_{U_{\xi \overline{w}}}] \ll [\mu_{U_{\overline{w}}}]$ by \cite[Proposition 8.4.5]{C(star)AlgebraByDixmier}. Using Lemma \ref{WeakDisjointnessOfSequences} we see that

\begin{alignat*}{2}
    &0 = \langle u, \xi\overline{w}\rangle_{\mathscr{H}} = \lim_{q\rightarrow\infty}\frac{1}{|F_{N_q'}|}\sum_{g \in F_{N_q'}}\int_Xu(g)w(g)\overline{\xi(g)}d\mu\\
    =&\int_X\lim_{q\rightarrow\infty}\frac{1}{|F_{N_q'}\setminus F_{N_{q-1}'}|}\sum_{g \in F_{N_q'}\setminus F_{N_{q-1}'}}u(g)w(g)\frac{\overline{\xi_q'}}{||\xi_q'||_2}d\mu = \lim_{q\rightarrow\infty}\int_X||\xi_q'||_2d\mu = \lim_{q\rightarrow\infty}||\xi_q'||_2 = \epsilon > 0,
\end{alignat*}
which yields the desired contradiction.
\end{proof}

\begin{corollary}\label{StrongDisjointnessOfSequencesCorollary}
    Let $G$ be a countably infinite
amenable group and $(F_N)_{N \in \mathbb{N}}$ a right-F\o lner sequence in $G$. Let $\mathcal{H}$ be a Hilbert space, let $u:G\rightarrow\mathcal{H}$ be bounded, and let $c:G\rightarrow\mathbb{C}$ bounded.
\begin{enumerate}[(i)]
    \item If $u$ is spectrally Lebesgue, and $c$ is spectrally singular, then

\begin{equation}\label{StronglyOrthogonalEquation3}
    \lim_{n\rightarrow\infty}\left|\left|\frac{1}{|F_n|}\sum_{g \in F_n} u(g)c(g)\right|\right| = 0.
\end{equation}

    \item If $u$ is nearly weakly mixing and $c$ is compact, then we again have Equation \eqref{StronglyOrthogonalEquation3}.
\end{enumerate}
\end{corollary}

\begin{proof}
    Since $\mathcal{H}$ is a separable Hilbert space, let us assume that it is infinite dimensional and let $\{e_n\}_{n = 1}^\infty$ be an orthonormal basis for $\mathcal{H}$. Consider the map $i:\mathcal{H}\rightarrow \ell^2(\mathbb{N},\nu)\cap \ell^\infty(\mathbb{N},\nu)$ where $\nu$ is the counting measure, given by $i(\xi) = (\langle \xi,e_n\rangle)_{n = 1}^\infty$. We see that $||i(\xi)||_\infty \le ||\xi||$, and that $i(\mathcal{H})$ is a Hilbert space embedding of $\mathcal{H}$ into $\ell^2(\mathbb{N},\nu)$. We now apply Theorem \ref{StrongDisjointnessOfSequences} to see that 
    
    \begin{equation}
    0 = \lim_{n\rightarrow\infty}\left|\left|\frac{1}{|F_n|}\sum_{g \in F_n} i(u(g))c(g)\right|\right|_2 = \lim_{n\rightarrow\infty}\left|\left|\frac{1}{|F_n|}\sum_{g \in F_n} u(g)c(g)\right|\right|.
\end{equation}
\end{proof}

\section{Measure preserving systems}
\subsection{Main results}

\begin{definition}
    Let $G$ be an amenable group and $\mathcal{X} := (X,\mathscr{B},\nu,(T_g)_{g \in G})$ a measure preserving $G$-system. The system $\mathcal{X}$ has \textbf{singular spectrum} if $[\mu_T]\perp[\mu_L]$. 
\end{definition}

\begin{lemma}\label{EquivalenceOfSpectralMeasuresLemma}
    Let $G$ be a countably infinite amenable group and $\mathcal{X} := (X,\mathscr{B},\nu,(T_g)_{g \in G})$ a measure preserving $G$-system, $f \in L^2(X,\nu)$, and $\xi = (T_gf)_{g \in G}$. We have $[\mu_{U_\xi}] = [\mu_{T_f}]$.\footnote{This lemma is the only instance in which $T_f$ refers to the notation after Theorem \ref{GeneralizedHahnHellinger} instead of the usual notation.}
\end{lemma}

\begin{proof}
    If suffices to observe that for any permissible pair $P := (\{U_g\xi\}_{g \in G},(N_q)_{q = 1}^\infty)$ and $\mathscr{H} := \mathscr{H}(P)$,

    \begin{equation}
        \langle U_h\xi,\xi\rangle_{\mathscr{H}} = \lim_{q\rightarrow\infty}\frac{1}{|F_{N_q}|}\sum_{g \in F_{N_q}}\int_XT_{gh}f_0T_g\overline{f_0}d\nu = \int_XT_hf_0\overline{f_0}d\nu = \langle T_hf, f\rangle,
    \end{equation}
    so $U_\xi$ is unitarily equivalent to $T_{f}$.
\end{proof}

\begin{theorem}\label{GeneralPrincipleForAmenableGroups}
    Let $G$ be a countably infinite amenable group and $(F_n)_{n = 1}^\infty$ a right-F\o lner sequence in $G$. Let $a_1,\cdots,a_\ell:G\rightarrow G$ be a collection of mappings for which the collection of mappings $A(h) := \{a_i(\cdot)\}_{i = 1}^\ell\cup\{a_i(\cdot h)\}_{i = 1}^\ell$ is jointly ergodic for every $h \in G\setminus\{e_G\}$. Let $(X,\mathscr{B},\nu)$ be a probability space and $T_g,S_g:X\rightarrow X$ be (not necessarily commuting) measure preserving actions of $G$ for which $(T_g)_{g \in G}$ has singular spectrum and $(S_g)_{g \in G}$ is ergodic. For any $f_0,f_1,\cdots,f_\ell \in L^\infty(X,\nu)$,

    \begin{equation}\label{GeneralConsequenceOfAmenableLebesgueSpectrumvdC}
        \lim_{N\rightarrow\infty}\frac{1}{|F_N|}\sum_{g \in F_N}T_gf_0S_{a_1(g)}f_1S_{a_2(g)}f_2\cdots S_{a_\ell(g)}f_\ell  = \mathbb{E}[f_0|\mathcal{I}_T]\prod_{j = 1}^\ell\int_Xf_jd\nu,
    \end{equation}
    with convergence taking place in $L^2(X,\nu)$. Furthermore, if $A(h)$ is only assumed to be jointly ergodic for totally ergodic systems, then Equation \eqref{GeneralConsequenceOfAmenableLebesgueSpectrumvdC} holds when $(S_g)_{g \in G}$ is totally ergodic.
\end{theorem}

\begin{proof}
    We proceed by induction on $\ell$. We see that the base case of $\ell = 0$ follows from the mean ergodic theorem for amenable groups. For the inductive step, we may use standard linearity arguments as well as the inductive hypothesis to reduce to the case in which $\int_Xf_1d\nu = 0.$ Let $h \in G\setminus\{e_G\}$ be arbitrary, and note that the joint ergodicity of $A(h)$ (for totally ergodic systems) implies that

    \begin{equation}
        \lim_{N\rightarrow\infty}\frac{1}{|F_N|}\sum_{g \in F_N}S_{a_1(gh)}f_1\cdots S_{a_\ell(gh)}f_\ell\overline{S_{a_1(g)}f_1\cdots S_{a_\ell(g)}f_\ell} d\nu(g) = 0,
    \end{equation}
    so we may apply Theorem \ref{LebesgueSpectrumAmenablevdC} to the function $u:G\rightarrow L^2(X,\nu)$ given by $u(g) = S_{a_1(g)}f_1\cdots S_{a_\ell(g)}f_\ell$ and deduce that $u$ is spectrally Lebesgue. Lemma \ref{EquivalenceOfSpectralMeasuresLemma} tells us that $(T_gf_0)_{g \in G}$ is spectrally singular, so the desired result now follows from Theorem \ref{StrongDisjointnessOfSequences}(i).
\end{proof}

It is clear that Theorem \ref{GeneralPrincipleForAmenableGroups} implies Theorem \ref{GeneralPrincipleForLCAGroups}. To see that Theorem \ref{GeneralPrincipleForLCAGroups} implies Theorem \ref{Application1Intro} and \ref{Application2Intro} we only need to recall the following 2 results from \cite{JointErgodicityForLCAGroups}. We remark that the degree conditions appearing in Theorems \ref{Application1Intro} and \ref{Application2Intro} are only used to ensure that for any $h \in F\setminus\{0\}$, the family $A(h) = \{p_i(\cdot)\}_{i = 1}^\ell\cup\{p_i(\cdot h)\}_{i = 1}^\ell$ consists of independent polynomials. 

\begin{theorem}[{\cite[Theorem 1.2]{JointErgodicityForLCAGroups}}]\label{BestMoragues1}
    Let $F$ be a countable field with characteristic $0$. Let $(X,\mathscr{B},\mu,(T_g)_{g \in F})$ be an ergodic measure preserving $(F,+)$-system and $(F_n)_{n = 1}^\infty$ a F\o lner sequence in $(F,+)$. Let $p_1,\cdots,p_k \in F[x]$ be independent polynomials. Then for any $f_1,\cdots,f_k \in L^\infty(X,\mu)$ we have

    \begin{equation}
        \lim_{N\rightarrow\infty}\frac{1}{\left|F_N\right|}\sum_{n \in F_N}\prod_{j = 1}^kT_{p_j(n)}f_j = \prod_{j = 1}^k\int_Xf_jd\mu
    \end{equation}
    with convergence taking place in $L^2(X,\mu)$.
\end{theorem}

\begin{theorem}[{\cite[Theorem 1.6]{JointErgodicityForLCAGroups}}]\label{BestMoragues2}
    Let $R$ be a good ring. Let $(X,\mathscr{B},\mu,(T_g)_{g \in R})$ be a totally ergodic measure preserving $(R,+)$-system. Let $(F_n)_{n = 1}^\infty$ be a F\o lner sequence in $(R,+)$ and let $p_1,\cdots,p_k \in R[x]$ be independent polynomials. Then for any $f_1,\cdots,f_k \in L^\infty(X,\mu)$ we have

    \begin{equation}
        \lim_{N\rightarrow\infty}\frac{1}{\left|F_N\right|}\sum_{n \in F_N}\prod_{j = 1}^kT_{p_j(n)}f_j = \prod_{j = 1}^k\int_Xf_jd\mu
    \end{equation}
    with convergence taking place in $L^2(X,\mu)$.
\end{theorem}

\begin{theorem}\label{Application3Main}
    Let $G$ be a countably infinite amenable group and $(F_n)_{n = 1}^\infty$ a right-F\o lner sequence in $G$. Let $(X,\mathscr{B},\nu)$ be a probability space and $T_g,S_g:X\rightarrow X$ measure preserving actions of $G$ for which the action $(T_g)_{g \in G}$ has singular spectrum. Let $a:G\rightarrow G$ be a map such that for every primary representation $\pi \in \text{\v{G}}\setminus\{0\}$ 
    and $h \in G\setminus\{e_G\}$ we have
    \begin{equation}\label{GeneralizedUniformDistributionOfDifferencesAssumptionEquation}
        \lim_{N\rightarrow\infty}\frac{1}{|F_N|}\sum_{g \in F_N}\pi\left(a(gh)a(g)^{-1}\right) = 0,
    \end{equation}
    with convergence taking place in the strong operator topology.
    \begin{enumerate}[(i)]
        \item For any $f_1,f_2 \in L^\infty(X,\mu)$ we have

        \begin{equation}
            \lim_{N\rightarrow\infty}\frac{1}{|F_N|}\sum_{g \in F_N}T_gf_1S_{a(g)}f_2 = \mathbb{E}[f_1|\mathcal{I}_T]\mathbb{E}[f_2|\mathcal{I}_S]
        \end{equation}
        with convergence taking place in $L^2(X,\mu)$.

        \item If $A \in \mathscr{B}$, then

        \begin{equation}
            \lim_{N\rightarrow\infty}\frac{1}{|F_N|}\sum_{g \in F_N}\nu\left(A\cap T_gA\cap S_{a(g)}A\right) \ge \nu(A)^3.
        \end{equation}

        \item Now assume that $G$ is finitely generated and that Equation \eqref{GeneralizedUniformDistributionOfDifferencesAssumptionEquation} only holds for $\pi \in \text{\v{G}}$ that do not vanish on any subgroup $H \le G$ of finite index. Then (i) and (ii) still hold when the action $(S_g)_{g \in G}$ is totally ergodic.
    \end{enumerate}
\end{theorem}

\begin{proof}
    We begin by proving item $(i)$. Let $f_2' = f_2-\mathbb{E}[f_2|\mathcal{I}_S]$ and let $S = \int^{\oplus}\pi_pd\mu_S(p)$ be the central decomposition of $S$. Let $f \in L^2(X,\nu)$, let $f' \in \int^{\oplus}H_pd\mu_S(p)$ be the image of $f$, and assume that $||f'(p)||_p \in L^\infty(\text{\v{G}},\mu_S)$. Note that $f'(0) = 0$. We see that for $h \in G\setminus\{e_G\}$, we have

    \begin{alignat*}{2}
        &\left|\left|\lim_{N\rightarrow\infty}\frac{1}{|F_N|}\sum_{g \in F_N}S_{a(gh)a(g)^{-1}}f\right|\right| = \left|\left|\lim_{N\rightarrow\infty}\frac{1}{|F_N|}\sum_{g \in F_N}\int^{\oplus}\pi_p(a(gh)a(g)^{-1})f'd\mu_S(p)\right|\right|\\
        =&\left|\left|\int^{\oplus}\lim_{N\rightarrow\infty}\frac{1}{|F_N|}\sum_{g \in F_N}\pi_p(a(gh)a(g)^{-1})f'd\mu_S\right|\right| = \left|\left|\int^{\oplus}0d\mu_S\right|\right| = 0,
    \end{alignat*}
    Standard approximation arguments show that the previous calculations hold for any $f \in L^2(X,\nu)$, so Theorem \ref{LebesgueSpectrumAmenablevdC} tells us that $(S_{a(g)}f_2')_{g \in G}$ will be a spectrally Lebesgue sequence. Lemma \ref{EquivalenceOfSpectralMeasuresLemma} tells us that $(T_gf_1)_{g \in G}$ is spectrally singular, so Theorem \ref{StrongDisjointnessOfSequences}(i) tells us that

    \begin{equation}
            \lim_{N\rightarrow\infty}\frac{1}{|F_N|}\sum_{F_N}T_gf_1S_{a(g)}f_2' = 0.
        \end{equation}
    The desired result now follows from the mean ergodic theorem.

    We now proceed to prove (ii). Using \cite[Lemma 1.6]{MRf2CT} we see that for any bounded nonnegative $h \in L^2(X,\mu)$ we have

\begin{equation}
    \int_Xh\cdot\mathbb{E}[h|\mathcal{I}_T]\cdot\mathbb{E}[h|\mathcal{I}_S]d\mu \ge \left(\int_Xhd\mu\right)^3,
\end{equation}
so the desired result follows from part (i) after setting $h = f_1 = f_2 = \mathbbm{1}_A$.

Lastly, we prove (iii). Since $G$ is finitely generated, a well-known theorem of M. Hall \cite{HallsTheoremOnFiniteIndexSubgroups} tells us that for each $n \in \mathbb{N}$, $G$ has finitely many subgroups $H$ with $[G:H] = n$, so $G$ has countably many subgroups of finite index. If $\pi$ is a representation of $G$ that vanishes on the subgroup $H$, then $H$ is normal and $\pi$ is a representation of $G/H$. When $G/H$ is finite, there are finitely many quasi-equivalence classes of primary representations of $G/H$. Consequently, there are only countably many $\pi \in \text{\v{G}}$ that vanish on some subgroup of finite index, and we denote the set of such $\pi$ by $I$. In order to repeat the proof of (i), we only need to show that $\mu_S(I) = 0$, so let us assume for the sake of contradiction that this is not the case. Since $I$ is countable, there is some $\pi \in I$ for which $\mu_S(\{\pi\}) > 0$, so \cite[Proposition 8.4.5]{C(star)AlgebraByDixmier} tells us that $\pi$ is quasi-equivalent to a subrepresentation of $S$, which contradicts the totally ergodicity of $S$.
\end{proof}

\subsection{Examples}\label{DirectSumZmodp}
The purpose of this subsection is to create the examples discussed at the end of Section 1. Throughout this subsection $p$ will denote an odd prime and $G = \oplus_{n = 1}^\infty\left(\mathbb{Z}/p\mathbb{Z}\right)$.\\

\noindent\textbf{Example 1:} Viewing $G$ as $(\overline{\mathbb{F}_p},+)$, let $a:\overline{\mathbb{F}_p}\rightarrow \overline{\mathbb{F}_p}$ be given by $a(g) = g^2$. For $n \in \mathbb{N}$, consider

\begin{equation}
    F_n := \{(x_i)_{i = 1}^\infty \in G\ |\ x_i = 0\ \forall\ i > n\},
\end{equation}
and observe that for any $(b_n)_{n = 1}^\infty \subseteq G$ we have that $(F_n+b_n)_{n = 1}^\infty$ is a F\o lner sequence in $G$. Now let $h \in G\setminus\{e_G\}$ and $\chi \in \widehat{G}\setminus\{e_{\widehat{G}}\}$ be arbitrary and observe that

\begin{alignat*}{2}
    &\lim_{N\rightarrow\infty}\frac{1}{p^N}\sum_{g \in F_N+b_N}\chi((g+h)^2-g^2) = \lim_{N\rightarrow\infty}\frac{\chi(2hb_N+h^2)}{p^N}\sum_{g \in F_N}\chi(2hg)\\
    = &\lim_{N\rightarrow\infty}\frac{\chi(2hb_N+h^2)}{p^N}\prod_{n = 1}^N\left(\sum_{i = 0}^{p-1}\chi(2h(\underbrace{0,\cdots,0,i}_{n},0,\cdots))\right) = 0,
\end{alignat*}
where the last equality follows from the fact that multiplication by $2h$ bijects $\overline{\mathbb{F}_p}$ to itself, so there is some $n \in \mathbb{N}$ for which 

\begin{equation*}
    \chi(2h(\underbrace{0,\cdots,0,i}_{n},0,\cdots)) \neq 1\text{, hence }\sum_{i = 0}^{p-1}\chi(2h(\underbrace{0,\cdots,0,i}_{n},0,\cdots)) = \sum_{i = 0}^{p-1}\chi(2h(\underbrace{0,\cdots,0,1}_{n},0,\cdots))^i = 0,
\end{equation*}
where the last equality follows from the observation that the range of $\chi$ is contained in the $p$th roots of unity.\\

\noindent\textbf{Example 2:} Now let us consider $G$ as $(\mathbb{F}_p[t],+)$. Let us proceed as in \cite{AsymptoticTotalErgodicityOfFpActions} and define an additive action $S$ of $G = \mathbb{F}_p[t]$ to be totally ergodic if for any $m \in \mathbb{F}_p[t]\setminus\{0\}$ the action $(S_{mg})_{g \in G}$ is ergodic. We remind the reader that this definition of total ergodicity depends on the ring structure that we imposed on $G$, and it differs from the definition given in Section 1.2. Proposition 1.14 and Example 1.16 of \cite{AsymptoticTotalErgodicityOfFpActions} tell us that for any separable polynomial\footnote{If $R$ is a ring of characteristic $p$, then a nonconstant polynomial $r(y) \in R[y]$ is \textit{separable} if $r(y) = b_0+\sum_{n = 1}^Na_ny^n$ and $a_n = 0$ when $p|n$.} $p(y) \in \left(\mathbb{F}_p[t]\right)[y]$, any totally ergodic $\mathbb{F}_p[t]$-system $(X,\mathscr{B},\nu,(S_g)_{g \in \mathbb{F}_p[t]})$, any $f \in L^2(X,\nu)$, and any F\o lner sequence $(F_n)_{n = 1}^\infty$, we have

\begin{equation}
    \lim_{N\rightarrow\infty}\frac{1}{|F_N|}\sum_{g \in F_N}S_gf = \int_Xfd\nu.
\end{equation}
Now suppose that $p(y) \in \left(\mathbb{F}_p[t]\right)[y]$ is separable and $p'(y)$ is separable, such as when $p(y) = y^2$.\footnote{We remind the reader that this is a fundamentally different example than Example 1. While we are considering $p(y) = y^2$ in both examples, the multiplication is different.} In this case we have that $p(g+h)-p(g)$ is a separable polynomial in $g$ for any $h \neq 0$. We see that for any $f \in L^2(X,\nu)$ satisfying $\int_Xfd\nu = 0$ and $h \in \mathbb{F}_p[t]\setminus\{0\}$, we have

\begin{equation}
    \lim_{N\rightarrow\infty}\frac{1}{|F_N|}\sum_{g \in F_N}\langle S_{p(g+h)}f,S_{p(g)}f\rangle = \lim_{N\rightarrow\infty}\frac{1}{|F_N|}\sum_{g \in F_N}\int_XS_{p(g+h)-p(g)}f\overline{f}d\nu = \left|\int_Xfd\nu\right|^2 = 0.
\end{equation}
Consequently, we see that we can apply Theorem \ref{LebesgueSpectrumAmenablevdC} to see that $(S_{p(g)}f)_{g \in G}$ is a spectrally Lebesgue. It follows that if $(T_g)_{g \in G}$ has singular spectrum, then for any $f' \in L^{\infty}(X,\nu)$ and $A \in \mathscr{B}$,

\begin{equation}
    \lim_{N\rightarrow\infty}\left|\left|\frac{1}{|F_N|}\sum_{g \in F_N}T_gf'S_{p(g)}f\right|\right| = 0\text{ and }\lim_{N\rightarrow\infty}\frac{1}{|F_N|}\sum_{g \in F_N}\nu\left(A\cap T_gA\cap S_{p(g)}A\right) \ge \nu(A)^3.
\end{equation}
\bibliographystyle{abbrv}
\begin{center}
	\bibliography{references}
\end{center}
\end{document}